\documentclass[oneside,english]{amsart}
\usepackage{amsthm}
\usepackage{amstext}
\usepackage{amssymb}
\usepackage{esint}

\makeatletter
\numberwithin{equation}{section}
\numberwithin{figure}{section}
\theoremstyle{plain}
\newtheorem{thm}{Theorem}[section]
 \theoremstyle{definition}
  \newtheorem{example}[thm]{Example}
\makeatother

\usepackage{babel}

\begin{document}

\title{A Note on Composition Operators in a Half-Plane}

\author{Hari Bercovici and Dan Timotin}


\thanks{HB was supported in part by grants from the National Science Foundation.
DT was supported in part by a grant of the Romanian National Authority
for Scientific Research, CNCS--UEFISCDI, project number PN-II-ID-PCE-2011-3-0119. }

\address{HB: Department of Mathematics, Indiana University, Bloomington, IN
47405, USA}

\email{bercovic@indiana.edu}

\address{DT: Simion Stoilow Institute of Mathematics of the Romanian Academy,
PO Box 1-764, Bucharest 014700, Romania}

\email{Dan.Timotin@imar.ro}
\begin{abstract}
Conditions for a composition operator on the Hardy space of the disk
to have closed range or be similar to an isometry are well known.
We provide such conditions for composition operators on the Hardy
space of the upper half-plane. We also show that the operator of composition
with an analytic self-map $\Phi$ of the upper half-plane can be similar
to an isometry even when $\Phi$ is far from being an inner function.
\end{abstract}
\maketitle

\section{Introduction}

We denote by $\mathbb{C}^{+}=\{x+iy:y>0\}$ the upper half of the
complex plane $\mathbb{C}$, and by $H_{\mathbb{C}^{+}}^{2}$ the
corresponding Hardy space. Thus, an analytic function $u:\mathbb{C}^{+}\to\mathbb{C}$
belongs to $H_{\mathbb{C}^{+}}^{2}$ when\[
\|u\|_{2}^{2}=\sup_{y>0}\int_{-\infty}^{\infty}|u(x+iy)|^{2}\, dx<\infty.\]

Consider an analytic function $\Phi:\mathbb{C}^{+}\to\mathbb{C}^{+}$,
and use the notation\[
C_{\Phi}u=u\circ\Phi\]
when $u$ is defined on $\mathbb{C}^{+}$. The functions $\Phi$ for
which $C_{\Phi}$ is a bounded operator on $H_{\mathbb{C}^{+}}^{2}$
were characterized by Matache \cite{matache}. To explain his result,
we write $\Phi$ in Nevanlinna form:\begin{equation}
\Phi(z)=\alpha+\beta z+\int_{-\infty}^{\infty}\frac{1+tz}{t-z}\, d\rho(t),\quad z\in\mathbb{C}^{+},
\label{eq:Nevanlinna-rep}\end{equation}
where $\alpha\in\mathbb{R},\beta\ge0,$ and $\rho$ is a finite, positive
Borel measure on the real line $\mathbb{R}$. Then $C_{\Phi}$ is
a bounded operator on $H_{\mathbb{C}^{+}}^{2}$ if and only if $\beta>0$,
in which case, according to the results of Elliott and Jury \cite{elliott-jury}
(see also \cite{jury-norms} for related results), the norm, essential
norm, and spectral radius of $C_{\Phi}$ are equal to $\beta^{-1/2}$.
In this note we restrict ourselves to the case in which $\beta=1$,
so that $C_{\Phi}$ is a contraction. In this case, it is known that
$C_{\Phi}$ is an isometry if and only if $\rho$ is singular relative
to Lebesgue measure. This follows from the work of Letac \cite{letac-proc}
and Chalendar and Partington \cite{chal-part}.

Based on the work of Bayart \cite{bayart} and Nordgren \cite{nordgren}
in the unit disk, it may be reasonable to surmise along with Elliott
\cite{elliott-sim} that $C_{\Phi}$ is not similar to an isometry unless
it is already an isometry. This is however not correct: we show that
there exist functions $\Phi$ for which $\rho$ is absolutely continuous
and $C_{\Phi}$ is similar to an isometry. We also provide criteria
for $C_{\Phi}$ to have closed range in terms of an associated family
of probability measures on $\mathbb{R}$ which are the analogues of
the so-called Alexandrov-Clark measures studied in the context of
the unit disk. (These measures also appear in \cite{letac-mal} as
the transition measures of a Markov process.) Another criterion for
closed range is closely related with results of Cima, Thomson, and
Wogen \cite{cima-wogen}.

\section{Preliminaries\label{sec:Preliminaries}}

Consider an analytic function $\Phi:\mathbb{\mathbb{C}}^{+}\to\mathbb{C}^{+}$
given by the Nevanlinna representation

\begin{equation}
\Phi(z)=\alpha+z+\int_{-\infty}^{\infty}\frac{1+tz}{t-z}\, d\rho(t),\quad z\in\mathbb{C}^{+},
\label{eq:beta-0}\end{equation}
where $\alpha\in\mathbb{R}$ and $\rho$ is a finite, positive Borel
measure on $\mathbb{R}$.  It is well known that the limits\[
\Phi(x)=\lim_{y\downarrow0}\Phi(x+iy)\]
exist for almost every $x\in\mathbb{R}$, relative to Lebesgue measure.
Moreover, the measure $\rho$ is singular if and only if $\Phi(x)$
is real for almost every $x$. The result of Letac \cite{letac-proc}
mentioned in the introduction is as follows. We use the notation $|\sigma|$
for the Lebesgue measure of a Borel subset $\sigma\subset\mathbb{R}$.
\begin{thm}
\label{thm:letac}With the above notation, assume that $\rho$ is
singular relative to Lebesgue measure. Then the map $x\mapsto\Phi(x)$
is a measure preserving transformation of the real line. In other
words,\[
|\{x\in\mathbb{R}:\Phi(x)\in\sigma\}|=|\sigma|\]
for every Borel set $\sigma\subset\mathbb{R}$.
\end{thm}
This result implies, of course, that $C_{\Phi}$ is an isometry when
$\rho$ is singular. 

Assume that $\Phi$ is defined by (\ref{eq:beta-0}), and the functions $G_{\tau}:\mathbb{C}^{+}\to\mathbb{C}^{+}$
defined by\[
G_{\tau}(z)=\frac{1}{\tau-\Phi(z)},\quad z\in\mathbb{C}^{+},\tau\in\mathbb{R}.\]
The Nevanlinna representation (\ref{eq:Nevanlinna-rep}) applies to
$G_{\tau}$, but can be simplified because of the asymptotic properties
of this function. More precisely, for each $\tau\in\mathbb{R}$
there exists a Borel probability measure $\mu_{\tau}$ on $\mathbb{R}$
such that \[
G_{\tau}(z)=\int_{\infty}^{\infty}\frac{d\mu_{\tau}(t)}{t-z},\quad z\in\mathbb{C}^{+}.\]

Now, the fact that $\Phi(x)$ is real for almost every $x$ is equivalent
to saying that $G_{\tau}(x)$ is real for almost every $x$, and therefore
$\rho$ is singular if and only if $\mu_{\tau}$ is singular for every
$\tau\in\mathbb{R}$. The following result, essentially equivalent to Letac's
theorem, was proved by Hru\v s\v c\"ev and Vinogradov \cite{hru-vin}
for the Cauchy transforms of singular probability measures.
\begin{thm}
\label{thm:hruscev}Let $\mu$ be a Borel probability measure on $\mathbb{R}$,
singular relative to Lebesgue measure, and let \[
G(z)=\int_{-\infty}^{\infty}\frac{d\mu(t)}{t-z},\quad z\in\mathbb{C}^{+},\]
denote its Cauchy transform. Then\[
|\{x\in\mathbb{R}:G(x)>y\}|=|\{x\in\mathbb{R}:G(x)<-y\}|=\frac{1}{y}\]
for every $y>0$.
\end{thm}
The special case of these results for measures ($\rho$ or $\mu$)
with finite support is due to Boole \cite{boole}. Theorem \ref{thm:hruscev}
has an extension to arbitrary probability measures due to Tsereteli
\cite{tsereteli} and Hru\v s\v c\"ev and Vinogradov \cite{hru-vin}.
To formulate this result, we use the Lebesgue decomposition of a probability
measure $\mu$ on $\mathbb{R}$ as \[
\mu=\mu_{\text{{\rm ac}}}+\mu_{\text{{\rm s}}},\]
where $\mu_{\text{{\rm ac}}}$ is absolutely continuous relative to
Lebesgue measure and $\mu_{\text{{\rm s}}}$ is singular.
\begin{thm}
\label{thm:tsiriteli}Let $\mu$ be a Borel probability measure on
$\mathbb{R}$, and let \[
G(z)=\int_{-\infty}^{\infty}\frac{d\mu(t)}{t-z},\quad z\in\mathbb{C}^{+},\]
denote its Cauchy transform. Then\[
\lim_{y\uparrow\infty}y|\{x\in\mathbb{R}:\Re G(x)>y\}|=\lim_{y\uparrow\infty}y|\{x\in\mathbb{R}:\Re G(x)<-y\}|=\mu_{\text{{\rm s}}}(\mathbb{R}).\]

\end{thm}

We refer to \cite{tamarkin} for the Nevanlinna representation, and to \cite{cima-book} for the results mentioned above.

\section{Operators with Closed Range\label{sec:Operators-with-Closed}}

We use the notations $\Phi,\alpha,\rho,\mu_{\tau}$ introduced in
Section \ref{sec:Preliminaries}. In particular, $\mu_{\tau,\text{{\rm s}}}$
denotes the singular summand of the probability measure $\mu_{\tau}$.
As pointed out in the introduction, the composition operator $C_{\Phi}$
is a contraction. We denote\[
\mathbf{A}=\inf_{u\in H_{\mathbb{C}^{+}}^{2}\setminus\{0\}}\frac{\|C_{\Phi}u\|_{2}^{2}}{\|u\|_{2}^{2}},\]
so that $C_{\Phi}$ has closed range if and only if $\mathbf{A}>0$.
Another constant related with $\Phi$ is defined as\[
\mathbf{B}=\inf_{-\infty<a<b<\infty}\frac{|\{x\in\mathbb{R}:\Phi(x)\in(a,b)\}|}{|b-a|}.\]

Given a finite interval $(a,b)\subset\mathbb{R}$, denote by $D_{a,b}$
the disk with diameter $(a,b)$, and define yet another constant\[
\mathbf{C}=\inf_{-\infty<a<b<\infty}\frac{|\{x\in\mathbb{R}:\Phi(x)\in D_{a,b}\}|}{|b-a|}.\]
 Finally, we set\[
\mathbf{D}=\inf_{\tau\in\mathbb{R}}\mu_{\tau,\text{{\rm s}}}(\mathbb{R}).\]

It may be worth noting that, in the definition of the constants $\mathbf{B}$
and $\mathbf{C}$, the infimum could be taken over intervals $(a,b)$
with length $b-a$ bounded by an arbitrarily small constant. This
is easy to see, and the relevant argument is contained in the proof
of the following result.
\begin{thm}
\label{thm:lower-bound}We have $\mathbf{A}=\mathbf{B}=\mathbf{C}=\mathbf{D}$.\end{thm}
\begin{proof}
Fix for the moment a finite interval $(a,b)$, and consider the function\[
\log\frac{z-b}{z-a},\quad z\in\mathbb{C}^{+},\]
where the principal branch of the logarithm is used, so that its imaginary
part is in $(0,\pi)$. In fact, \[
\Im\log\frac{z-b}{z-a}=\theta_{z},\]
where $\theta_{z}$ is the angle at $z$ subtended by $(a,b)$. It
follows that for $c>0$, the function\[
\exp\left(-ic\log\frac{z-b}{z-a}\right)\]
is bounded on $\mathbb{C}^{+}$. More precisely,\[
\left|\exp\left(-ic\log\frac{z-b}{z-a}\right)\right|\le\begin{cases}
\exp(\pi c) & \text{for }z\in\mathbb{C}^{+}\cap D_{a,b},\\
\exp(\pi c/2) & \text{for }z\in\mathbb{C}^{+}\setminus D_{a,b}.\end{cases}\]
It follows that the function\[
u_{c}(z)=\frac{\exp\left(-ic\log\frac{z-b}{z-a}\right)}{z+i},\quad z\in\mathbb{C}^{+},\]
belongs to $H_{\mathbb{C}^{+}}^{2}$. We have \[
\frac{\|C_{\Phi}u_{c}\|_{2}^{2}}{\|u_{c}\|_{2}^{2}}\le\frac{\int_{\Phi(x)\in D_{a,b}}\frac{\exp(2\pi c)}{1+|\Phi(x)|^{2}}\, dx+\int_{\Phi(x)\notin D_{a,b}}\frac{\exp(\pi c)}{1+|\Phi(x)|^{2}}\, dx}{\int_{(a,b)}\frac{\exp(2\pi c)}{1+x^{2}}\, dx+\int_{\mathbb{R}\setminus(a,b)}\frac{1}{1+x^{2}}\, dx},\]
and\[
\int_{\Phi(x)\notin(a,b)}\frac{1}{1+|\Phi(x)|^{2}}\, dx\le\int_{-\infty}^{\infty}\frac{1}{1+|\Phi(x)|^{2}}\, dx\le\int_{-\infty}^{\infty}\frac{1}{1+x^{2}}\, dx=\pi\]
because $C_{\Phi}$ is a contraction. Therefore\[
\frac{\|C_{\Phi}u_{c}\|_{2}^{2}}{\|u_{c}\|_{2}^{2}}\le\frac{\int_{\Phi(x)\in D_{a,b}}\frac{1}{1+|\Phi(x)|^{2}}\, dx+\pi\exp(-\pi c)}{\int_{(a,b)}\frac{1}{1+x^{2}}\, dx},\]
and letting $c$ tend to $+\infty$ we obtain\begin{eqnarray}
\mathbf{A} & \le & \frac{\int_{\Phi(x)\in D_{a,b}}\frac{1}{1+|\Phi(x)|^{2}}\, dx}{\int_{(a,b)}\frac{1}{1+x^{2}}\, dx}\nonumber \\
 & \le & \frac{\max\{1+t^{2}:t\in(a,b)\}}{\min\{1+|z|^{2}:z\in D_{a,b}\}}\frac{|\{x\in\mathbb{R}:F(x)\in D_{a,b}\}|}{b-a}.\label{eq:Aless}\end{eqnarray}
This inequality can now be improved to show that $\mathbf{A}\le\mathbf{C}$.
Indeed, divide the interval $(a,b)$ into two equal subintervals.
For one of these intervals, say $(a',b')$, we have\begin{equation}
\frac{|\{x\in\mathbb{R}:F(x)\in D_{a',b'}\}|}{b'-a'}\le\frac{|\{x\in\mathbb{R}:F(x)\in D_{a,b}\}|}{b-a}\label{eq:half-a-b}\end{equation}
so that inequality (\ref{eq:Aless}) applied to $(a',b')$, combined
with (\ref{eq:half-a-b}) yields\[
\mathbf{A}\le\frac{\max\{1+t^{2}:t\in(a',b')\}}{\min\{1+|z|^{2}:z\in D_{a',b'}\}}\frac{|\{x\in\mathbb{R}:\Phi(x)\in D_{a,b}\}|}{b-a}.\]
Repeating this operation, the first fraction can be made arbitrarily
close to 1, thus yielding $\mathbf{A}\le\mathbf{C}$.

Next we observe that $\{x\in\mathbb{R}:\Phi(x)\in(a,b)\}\subset\{x\in\mathbb{R}:\Phi(x)\in D_{a,b}\}$,
and this yields the inequality $\mathbf{B}\le\mathbf{C}$. On the
other hand, let $\varepsilon$ be an arbitrary positive number, and
choose $(a,b)$ so that \[
|\{x\in\mathbb{R}:\Phi(x)\in(a,b)\}|<(\mathbf{B}+\varepsilon)(b-a).\]
 There is then $\delta>0$ such that\[
|\{x\in\mathbb{R}:\Re\Phi(x)\in(a,b),\Im\Phi(x)<\delta\}|<(\mathbf{B}+\varepsilon)(b-a).\]
Divide now $(a,b)$ into $N$ equal parts, all of them of length $<\delta$,
and choose one of these intervals, say $(a',b')$, such that $|\{x\in\mathbb{R}:\Phi(x)\in D_{a',b'}\}|$
is the smallest. Then\begin{eqnarray*}
|\{x\in\mathbb{R}\colon\Phi(x)\in D_{a',b'}\}| & \le & \frac{1}{N}|\{x\in\mathbb{R}:\Re\Phi(x)\in(a,b),\Im\Phi(x)<\delta\}|\\
 & < & (\mathbf{B}+\varepsilon)(b'-a'),\end{eqnarray*}
so that $\mathbf{C}\le\mathbf{B}+\varepsilon$, and this yields $\mathbf{C}\le\mathbf{B}$
as $\varepsilon\to0$.

The inequality $\mathbf{B}\le\mathbf{A}$ follows immediately if we
prove that\[
\int_{F(x)\in\mathbb{R}}v(\Phi(x))\, dx\ge\mathbf{B}\int_{\mathbb{R}}v(x)\, dx\]
for an arbitrary positive integrable function $v$. This is trivially
verified when $v$ is a linear combination with positive coefficients
of functions of the form $\chi_{(a,b)}$, and the general case follows
by standard approximation procedures.

It remains to show that the constant $\mathbf{A}=\mathbf{B}=\mathbf{C}$
also equals $\mathbf{D}$, and this is where we use Theorem \ref{thm:tsiriteli}.
Fix $\varepsilon>0$, and choose $\tau\in\mathbb{R}$ such that\[
\mu_{\tau,\text{s}}(\mathbb{R})<\mathbf{D}+\varepsilon.\]
Theorem \ref{thm:tsiriteli} yields a positive number $y$ such that\[
|\{x\in\mathbb{R}:\Re G_{\tau}(x)<-y\}|<\frac{\mathbf{D}+\varepsilon}{y}.\]
Equivalently,\[
|\{x\in\mathbb{R}:\Phi(x)\in D_{a,b}\}|<(\mathbf{D}+\varepsilon)(b-a),\]
where $a=\tau$ and $b-a=1/y$. Letting $\varepsilon\to0$ we obtain
$\mathbf{C}\le\mathbf{D}$. For the opposite inequality, choose an
interval $(a_{0},b_{0})$ such that\[
|\{x\in\mathbb{R}:\Phi(x)\in D_{a_{0},b_{0}}\}|<(\mathbf{C}+\varepsilon)(b_{0}-a_{0}).\]
We construct intervals intervals $(a_{n},b_{n})$ such that each of
them is one half of $(a_{n-1},b_{n-1})$ and\[
|\{x\in\mathbb{R}:\Phi(x)\in D_{a_{n},b_{n}}\}|<(\mathbf{C}+\varepsilon)(b_{n}-a_{n})\]
 for every $n\ge1$. Denote by $\tau$ the common limit of the sequences
$(a_{n})_{n=0}^{\infty}$ and $(b_{n})_{n=0}^{\infty}$. Note that
$\tau\in[a_{n},b_{n}]$ for every $n$, so that $\tau$ divides this
interval into at most two subintervals one of which, say $(a'_{n},b'_{n})$
must also satisfy\[
|\{x\in\mathbb{R}:\Phi(x)\in D_{a'_{n},b'_{n}}\}|<(\mathbf{C}+\varepsilon)(b'_{n}-a'_{n}).\]
One of the sets $\{n:a'_{n}=\tau\}$ and $\{n:b'_{n}=\tau\}$ must
be infinite. For definiteness, assume that the first one is infinite,
so passing to a subsequence and relabeling, we obtain a sequence of
numbers $y_{n}\uparrow\infty$ such that\[
|\{x\in\mathbb{R}:\Phi(x)\in D_{\tau,\tau+1/y_{n}}\}|<\frac{\mathbf{C}+\varepsilon}{y_{n}},\quad n\ge1.\]
This is then equivalent to\[
|\{x\in\mathbb{R}:\Re G_{\tau}(x)<-y_{n}\}|<\frac{\mathbf{C}+\varepsilon}{y_{n}},\]
and this implies that $\mu_{\tau,\text{s}}(\mathbb{R})<\mathbf{C}+\varepsilon.$
The desired inequality $\mathbf{D}\le\mathbf{C}$ follows again by
letting $\varepsilon\to0$.
\end{proof}
\begin{example}
Here are a few illustrations of the preceding result. 
\begin{enumerate}
\item The function $\Phi(z)=\sqrt{z^{2}-1}$ has the property that $G_{0}(z)=-1/\sqrt{z^{2}-1}$
is represented by an absolutely continuous probability measure, namely
$dt/\pi\sqrt{1-t^{2}}$ on $(-1,1)$. It follows that $C_{\Phi}$
does not have closed range. 
\item For the function $\Phi(z)=z+\log z$, the measure $\mu_{\tau}$ has
a singular part for every $\tau\in\mathbb{R}$. More precisely, if
$\tau=t+\log t$, then singular part of $\mu_{\tau}$ consists of
a single atom at $t$ with mass $t/(1+t)$. These masses tend to zero
as $t\to0$, hence $C_{\Phi}$ does not have closed range. 
\item Consider now $\Phi(z)=z+\log((z-1)/(z+1))$. For this function, $\mu_{\tau,\text{s}}$
is supported by two points, and it is fairly easy to verify that its
mass is bounded away from zero. Thus $C_{\Phi}$ does have closed
range.
\end{enumerate}
\end{example}

\section{Similarity to an Isometry\label{sec:Examples}}

In this section we consider functions \[
\Phi(z)=\alpha+z+\int_{-\infty}^{\infty}\frac{1+tz}{t-z}\, d\rho(t),\quad z\in\mathbb{C}^{+},\]
such that $C_{\Phi}$ is similar to an isometry. Using the fact that
$C_{\Phi}$ is a contraction, a result of B. Sz.-Nagy \cite{sz-nagy},
implies that $C_{\Phi}$ is similar to an isometry if and only if\[
\inf_{n\ge1}\inf_{u\in H_{\mathbb{C}^{+}}^{2}\setminus\{0\}}\frac{\|C_{\Phi}^{n}u\|_{2}^{2}}{\|u\|_{2}^{2}}>0.\]
Denote by $\Phi_{n}$ the composition of $n$ copies of $\Phi$. Theorem
\ref{thm:lower-bound} shows that this is equivalent to\begin{equation}
\inf_{n\ge1}\inf_{-\infty<a<b<\infty}\frac{|\{x\in\mathbb{R}:\Phi_{n}(x)\in(a,b)\}|}{b-a}>0.\label{eq:uniform-below}\end{equation}
 
\begin{thm}
\label{thm:similarity}With the above notation, assume that $\rho$
is supported on a finite interval $[c,d]$ and\[
\lim_{x\uparrow c}\Phi(x)>\lim_{x\downarrow d}\Phi(x).\]
Then $C_{\Phi}$ is similar to an isometry provided that $|\alpha|$
is sufficiently large.\end{thm}
\begin{proof}
We only consider positive constants $\alpha$. Negative values are
treated similarly. The function $\Phi(x)$ is increasing on the intervals
$(-\infty,c)$ and $(d,+\infty),$and\[
\lim_{|x|\to\infty}\frac{\Phi(x)}{x}=1.\]
Choose points $c_{1}<c$ and $d_{1}>d$ such that $\Phi(c_{1})=\Phi(d_{1}).$
If $\alpha$ is sufficiently large, we have $\Phi(d_{1})>d_{1}$.
Setting $\eta=\Phi(d_{1})-d_{1}$, the following inequality\[
\Phi(x)\ge x+\eta,\quad x\in\mathbb{R}\setminus(c_{1},d_{1})\]
is satisfied, as can be seen by observing that $\Phi$ is convex on
$(-\infty,c)$ and concave on $(d,+\infty)$. We show that this is
sufficient to insure that $C_{\Phi}$ is similar to an isometry. Indeed,
given any point $t\in\mathbb{R}$, we can find a sequence $(t_{n})_{n=1}^{\infty}\subset\mathbb{R}\setminus(c_{1},d_{1})$
with the property that $\Phi(t_{1})=t,$ $\Phi(t_{n+1})=t_{n}$ for
$n\ge1$, and $t_{n}-t_{n+1}\ge\eta$ for all $n$ with one possible
exception when $t_{n}\ge d_{1}$ and $t_{n+1}\le c_{1}$. Note that
\[
\Phi'(x)-1=\int_{c}^{d}\frac{1+t^{2}}{(t-x)^{2}}\, dt,\quad x\in\mathbb{R}\setminus[c,d],\]
so that\[
\Phi'(x)-1\le\frac{k}{\text{dist}(x,[c,d])^{2}}\]
for some $c>0$. It follows that\[
\prod_{n=1}^{\infty}\Phi'(t_{n})\le\prod_{n=0}^{\infty}\left(1+\frac{k}{\text{dist}(d_{1}+n\eta,[c,d])^{2}}\right)\left(1+\frac{k}{\text{dist}(c_{1}-n\eta,[c,d])^{2}}\right).\]
The last product is finite, and this easily implies the inequality
(\ref{eq:uniform-below}).
\end{proof}
As mentioned in the introduction, there are functions of the form
\[
\Phi(z)=\alpha+z+\int_{c}^{d}\frac{1+tz}{t-z}\, d\rho(t),\quad z\in\mathbb{C}^{+},
\]
which satisfy the hypothesis of Theorem \ref{thm:similarity}, and such that $\rho$ is absolutely continuous relative to Lebesgue measure. Such an example is provided by\[
\Phi(z)=\alpha+z+\log\frac{z-1}{z+1},\quad z\in\mathbb{C}^{+},\]
for which $c=-1$, $d=1$, and  $\rho=dt/\pi(1+t^2)$ on $[-1,1]$.
For this particular case, $C_{\Phi}$ is similar to an isometry whenever
$\alpha\ne0$. Indeed, when $\alpha>0$, one can use the fact that
this function has infinite limit at $-1$ to choose $c_{1}$ and $d_{1}$
so that $\Phi(d_{1})>d_{1}$. (When $\alpha<0$, the corresponding
condition is $\Phi(c_{1})<c_{1}$.)

Theorem \ref{thm:similarity} does not cover all cases in which $C_{\Phi}$
is similar to an isometry. There are many examples where the support
of $\rho$ is not compact, including for instance the case
of singular measures $\rho$.

The inequality required in the statement is an essential hypothesis,
as shown by the function $\Phi(z)=\sqrt{z^{2}-1}$ for which $\Phi(-1)=\Phi(1)=0$
and $C_{\alpha+\Phi}$ does not even have closed range for any $\alpha\in\mathbb{R}$. 

The condition $\Phi(d_{1})>d_{1}$ which appears in the argument is
essential in determining the values of $\alpha$ for which this proof
works. Note that if $\Phi(d_{1})\le d_{1}$ then $\Phi$
has a repelling fixed point $x_{0}\ge d_{1}$, and this makes the
verification of condition (\ref{eq:uniform-below}) more difficult.
Indeed, one must rely on other real preimages under $\Phi$ which
the point $x_{0}$ might have.

The function \[
\Phi(z)=\alpha+\sqrt{z^{2}-1}+\frac{1}{1-z},\quad z\in\mathbb{C}^{+},\]
satisfies the hypothesis of the theorem. In this example, $\Phi$
always has a repelling fixed point $x_{0}>1$ when $\alpha>0$, and
that fixed point has no other real preimages if $\alpha$ is too small.
The function $\Phi$ also has a repelling fixed point $x_{0}\le-1$
for $\alpha\in[-3/2,0)$, but then it is possible to choose $c_{1}<x_{0}$
and carry out the argument.

\end{document}